\documentclass[preprint,review,10pt]{amsart}

\newtheorem{theorem}{Theorem}[section]
\newtheorem{lemma}[theorem]{Lemma}

\theoremstyle{definition}

\newtheorem{example}[theorem]{Example}

\newtheorem{proposition}[theorem]{Proposition}
\newtheorem{corollary}[theorem]{Corollary}

\theoremstyle{remark}

\numberwithin{equation}{section}

\begin{document}

\title{ Variation of the  first   eigenvalue of $(p,q)$-Laplacian along the  Ricci-harmonic flow }

\author{Shahroud Azami}
\address{Department of Mathematics, Faculty of Sciences,
Imam Khomeini International University, Qazvin, Iran.
}

\email{azami@sci.ikiu.ac.ir,\,\, shahrood78@yahoo.com }


\subjclass[2010]{58C40; 53C43; 53C44}


\keywords{Laplace, Ricci flow, Harmonic map }

\begin{abstract}
In this paper, we study monotonicity for the first eigenvalue of a class of $(p,q)$-Laplacian. We find the first variation formula for the first eigenvalue
of $(p,q)$-Laplacian on a closed Riemannian manifold evolving by the Ricci-harmonic flow and construct various monotic quantities by imposing some conditions on initial manifold.
\end{abstract}
\maketitle
\section{Introduction}
On the other hand, the study on eigenvalue problem has received remarkable attention.  Recently, many mathematicians  considered the eigenvalue problem of geometric operators under various geometric flows, because it is a very powerful tool for the understanding Riemannain manifold. The fundamental study of this works began  when Perelman \cite{GP} showed that the functional
\begin{equation*}
F=\int_{M} (R+|\nabla f|^{2}) e^{-f}\,d\mu
\end{equation*}
is nondecreasing along the Ricci flow coupled to a backward heat-type equation, where $R$ is the scalar curvature with respect to the metric  $g(t)$ and $d\mu$ denotes the volume form  of the metric  $g(t)$.
The nondecreasing of the functional $F$ implies  that the first  eigenvalue of the geometric  operator $-4\Delta+R$ is nondecreasing under the Ricci flow.
Then, Li \cite{JFL} and Zeng with et'al \cite{FZ} extended the geometric operator $-4\Delta+R$ to the operator $-\Delta+cR$ and  studied the monotonicity of eigenvalues of the operator $-\Delta+cR$ along Ricci flow and the Ricci-Bourguignon flow, respectively.

Also,  in    \cite{AA, JY,  LZ} has been investigated the evolution for the first eigenvalue of $p$-Laplacian along the Ricci-harmonic flow, Ricci flow and  mth mean curvature flow, respectively. A generalization of $p$-Laplacian is a class of $(p,q)$-Laplacian  which has applications in physic  and related sciences such as non-Newtonnian fluids, pseudoplastics \cite{JD, AEL} and we introduce it in later section.

On the other hand, geometric flows for instance, Ricci-harmonic flow  have been a topic of active research interest in mathematics and physics. A geometric flow  is an evolution of a geometric structure. Let $M$ be a closed $m$-dimensional Riemannian manifold with a Riemannian metric $g_{0}$. Hamilton for the first time in $1982$ introduced  the Ricci flow  as follows
\begin{equation*}
\frac{\partial g(t)}{\partial t}=-2Ric(g(t)),\,\,\,\,\,\,\,\,\,g(0)=g_{0},
\end{equation*}
where $Ric$ is the Ricci tensor of $g(t)$. The Ricci flow has been proved to be a very useful tool to improve metrics in Riemannian geometry, when $M$ is  compact. Now, let $(M^{m}, g)$ and $(N^{n},\gamma)$ be closed Riemannain manifolds. By Nash's embedding theorem, assume that  $N$ is isometrically embedded into  Euclidean space $e_{N}:(N^{n},\gamma)\hookrightarrow \mathbb{R}^{n+1}$ for a sufficiently large $d$. We identify maps $\phi: M\to N$ with $e_{N}\circ \phi:M\to \mathbb{R}^{d}$. M\"{u}ller \cite{RM} considered a generalization of Ricci flow  as
\begin{equation}\label{f1}
\begin{cases}
\frac{\partial g(t)}{\partial t}=-2Ric(g(t))+2\alpha\nabla\phi \otimes\nabla\phi,& g(0)=g_{0},\\
\frac{\partial \phi}{\partial t}=\tau_{g}\phi&\phi(0)=\phi_{0},
\end{cases}
\end{equation}
where $\alpha$ is a positive coupling constant, $\phi(t)$ is a family of smooth maps from $M$ to some closed target manifold $N$ and  $\tau_{g}\phi$ is the intrinsic Laplacian of $\phi$ which denotes the  tension field of $\phi$ with respect to the evolving metric $g(t)$. This evolution equation system called Ricci flow coupled with harmonic map flow or $(RH)_{\alpha}$ flow for short. M\"{u}ller in \cite{RM} shown that  system (\ref{f1}) has unique solution with initial data $(g(0),\phi(0))=(g_{0},\phi_{0})$. Also, the normalized  $(RH)_{\alpha}$ flow defined as
\begin{equation}\label{f2}
\begin{cases}
\frac{\partial g(t)}{\partial t}=-2Ric(g(t))+2\alpha\nabla\phi \otimes\nabla\phi+\frac{2}{m}rg(t),& g(0)=g_{0},\\
\frac{\partial \phi}{\partial t}=\tau_{g}\phi&\phi(0)=\phi_{0},
\end{cases}
\end{equation}
where $r=\frac{\int_{M}(R-\alpha |\nabla \phi|^{2})d\mu }{\int_{M}d\mu}$ is the
average of $R-\alpha |\nabla \phi|^{2}$. Under this normalized flow, the
volume of the solution metrics remains constant in time.

\section{Preliminaries}
\subsection{Eigenvalues of $p$-Laplacian}
Let $(M,g)$ be a closed Riemannian manifold and  $f:M\longrightarrow \mathbb{R}$ be a smooth function on $M$ or $f\in W^{1,p}(M)$. The Laplace-Beltrami operator acting on a smooth function $f$ on $M$ is the divergence of gradient of $f$, written as
\begin{equation*}
\Delta f= div(grad  \,f)=\frac{1}{\sqrt{\det g}}\partial_{i}(\sqrt{\det g}\,\partial_{j}f )
\end{equation*}
where $\partial_{i}f=\frac{\partial f}{\partial x^{i}}$.
The $p$-Laplacian  of $f$  for $1<p<\infty$ is defined as
\begin{eqnarray}\label{R3}
\triangle_{p}f&=& div(|\nabla f|^{p-2}\nabla f)\\\nonumber
&=&|\nabla f|^{p-2}\Delta f+(p-2)|\nabla f|^{p-4}(Hess f)(\nabla
f,\nabla f)
\end{eqnarray}
where
\begin{equation*}
(Hess f)(X,Y)=\nabla(\nabla f)(X,Y)=Y.(X.f)-(\nabla_{Y}X).f,\,\,\,
\,X,Y\in \mathcal{X}(M)
\end{equation*}
and in local coordinate, we get
\begin{equation*}
(Hessf)(\partial_{i},\partial_{j})=\partial_{i}\partial_{j}f-\Gamma_{ij}^{k}
\partial_{k} f.
\end{equation*}
Notice that when $p=2$, $p$-Laplacian is the Laplace-Beltrami operator.
Let $(M^{n}, g)$ be a closed Riemannian manifold.
In this paper, we  consider the nonlinear system introduced in \cite{AEK}, that is
\begin{equation}\label{pq1}
\begin{cases}
\Delta_{p}u=-\lambda |u|^{\alpha}|v|^{\beta}v \,\,\, \text{in} \,\,M\\
\Delta_{q}v=-\lambda |u|^{\alpha}|v|^{\beta}u \,\,\, \text{in}  \,\,M\\
(u,v)\in W^{1,p}(M)\times W^{1,q}(M)
\end{cases}
\end{equation}
 where $p>1,\,\, q>1$ and $\alpha, \beta$ are real numbers satisfying
\begin{equation}\label{pq2}
\alpha>0,\,\beta>0,\,\,\,\, \frac{\alpha+1}{p}+\frac{\beta+1}{q}=1.
\end{equation}

In  (\ref{pq1}),
we say that
$\lambda$ is  an eigenvalue  whenever for some $u\in W_{0}^{1,p}(M)$ and $v\in W_{0}^{1,q}(M)$,
\begin{eqnarray}\label{pg1}
\int_{M}|\nabla u|^{p-2}<\nabla u, \nabla \phi>d\mu&=&\lambda \int_{M}|u|^{\alpha}|v|^{\beta}v\phi d\mu,\\\label{pg2}
\int_{M}|\nabla v|^{q-2}<\nabla v, \nabla \psi>d\mu&=&\lambda \int_{M}|u|^{\alpha}|v|^{\beta}u\psi d\mu,
\end{eqnarray}
where  $\phi\in W^{1,p}(M)$,  $\psi\in W^{1,q}(M)$ and $ W_{0}^{1,p}(M)$  is  the closure of $C_{0}^{\infty}(M)$ in Sobolev space $ W^{1,p}(M)$. The pair $(u,v)$ is called a eigenfunctions corresponding to eigenvalue $\lambda$. A first positive eigenvalue of (\ref{pq1}) obtained as
\begin{equation*}
inf\{A(u,v): (u,v)\in W_{0}^{1,p}(M)\times W_{0}^{1,q}(M),\,\,B(u,v)=1\}
\end{equation*}
where
\begin{eqnarray*}
A(u,v)&=&\frac{\alpha+1}{p}\int_{M}|\nabla u|^{p}d\mu+\frac{\beta+1}{q}\int_{M}|\nabla v|^{q}d\mu,\\
B(u,v)&=&\int_{M}|u|^{\alpha}|v|^{\beta}uvd\mu.
\end{eqnarray*}


Let  $(M^{m}, g(t), \phi(t) )$ be a solution of  the  $(RH)_{\alpha}$ flow (\ref{f1}) on
the smooth manifold $(M^{m}, g_{0}, \phi_{0})$  in the interval $[0,T)$ then
\begin{equation}\label{R6}
\lambda(t)=\frac{\alpha+1}{p}\int_{M}|\nabla u|^{p}d\mu_{t}+\frac{\beta+1}{q}\int_{M}|\nabla v|^{q}d\mu_{t},
\end{equation}
defines the evolution of an eigenvalue  of (\ref{pq1}),  under the
variation of $(g(t), \phi(t))$  where the eigenfunctions associated to $\lambda(t)$
is normalized that is $B(u,v)=1$.
Motivated by the above works, in this paper we will study the first eigenvalue of a class of $(p,q)$-Laplacian  (\ref{pq1}) whose metric satisfies the  $(RH)_{\alpha}$ flow. Throughout of paper we write   $\frac{\partial u}{\partial t}=\partial_{t} u=u'$, $\mathcal{S}=Ric _{g}-\alpha \nabla\phi\otimes\nabla\phi$, $\mathcal{S}_{ij}=Ric _{ij}-\alpha \nabla_{i}\phi\nabla_{j}\phi$ and $S=R-\alpha|\nabla \phi |^{2}$.
\section{Variation of $\lambda(t)$}
In this section, we will give some useful evolution formulas for $\lambda(t)$ under the Ricci-harmonic flow. Now, we give a useful statement  about the variation of the first eigenvalue  of   (\ref{pq1}) under the $(RH)_{\alpha}$ flow.
\begin{lemma}
If $g_{1}$ and $g_{2}$ are two metrics on Riemannian manifold $M^{m}$ which satisfy $(1+\epsilon)^{-1}g_{1}< g_{2}<(1+\epsilon)g_{1}$ then for any $p\geq q>1$, we have
\begin{equation*}
\lambda(g_{2})-\lambda(g_{1})\leq\left((1+\epsilon)^{\frac{p+m}{2}} -(1+\epsilon)^{-\frac{m}{2}}  \right) \lambda(g_{1})
\end{equation*}
in particular, $\lambda(t)$ is  a continues function respect to $t$-variable.
\end{lemma}
\begin{proof}
By direct computation we complete the proof of lemma. In local coordinate we have $d\mu=\sqrt{\det g}\,dx^{1}\wedge...\wedge dx^{m}$, therefore
 \begin{equation*}
(1+\epsilon)^{-\frac{m}{2}}d\mu_{g_{1}}< d\mu_{g_{2}}<(1+\epsilon)^{\frac{m}{2}}d\mu_{g_{1}}.
\end{equation*}
Let
\begin{eqnarray}\label{gg1}
G(g,u,v)=\frac{\alpha+1}{p}\int_{M}|\nabla u|_{g}^{p}d\mu_{g}+\frac{\beta+1}{q}\int_{M}|\nabla v|_{g}^{q}d\mu_{g},
\end{eqnarray}
then
\begin{eqnarray*}
&&\int_{M}|u|^{\alpha}|v|^{\beta}uvd\mu_{g_{1}}G(g_{2},u,v)-\int_{M}|u|^{\alpha}|v|^{\beta}uvd\mu_{g_{2}}G(g_{1},u,v)\\&=&
 \frac{\alpha+1}{p}\int_{M}|u|^{\alpha}|v|^{\beta}uvd\mu_{g_{1}}\left(\int_{M}|\nabla u|_{g_{2}}^{p}d\mu_{g_{2}}-\int_{M}|\nabla u|_{g_{1}}^{p}d\mu_{g_{1}}\right)\\&&
+ \frac{\alpha+1}{p}\left(\int_{M}|u|^{\alpha}|v|^{\beta}uvd\mu_{g_{1}}-\int_{M}|u|^{\alpha}|v|^{\beta}uvd\mu_{g_{2}} \right)\int_{M}|\nabla u|_{g_{1}}^{p}d\mu_{g_{1}}\\&&
+\frac{\beta+1}{q}\int_{M}|u|^{\alpha}|v|^{\beta}uvd\mu_{g_{1}}\left(\int_{M}|\nabla v|_{g_{2}}^{q}d\mu_{g_{2}}-\int_{M}|\nabla v|_{g_{1}}^{q}d\mu_{g_{1}}\right)\\&&
+ \frac{\beta+1}{q}\left(\int_{M}|u|^{\alpha}|v|^{\beta}uvd\mu_{g_{1}}-\int_{M}|u|^{\alpha}|v|^{\beta}uvd\mu_{g_{2}} \right)\int_{M}|\nabla v|_{g_{1}}^{q}d\mu_{g_{1}}\\&\leq&
 \frac{\alpha+1}{p}\left((1+\epsilon)^{\frac{p+m}{2}} -(1+\epsilon)^{-\frac{m}{2}}  \right)\int_{M}|u|^{\alpha}|v|^{\beta}uvd\mu_{g_{1}}\int_{M}|\nabla u|_{g_{1}}^{p}d\mu_{g_{1}}\\&&
+\frac{\beta+1}{q}\left((1+\epsilon)^{\frac{q+m}{2}} -(1+\epsilon)^{-\frac{m}{2}}  \right)\int_{M}|u|^{\alpha}|v|^{\beta}uvd\mu_{g_{1}}\int_{M}|\nabla v|_{g_{1}}^{q}d\mu_{g_{1}}\\&\leq&
\left((1+\epsilon)^{\frac{p+m}{2}} -(1+\epsilon)^{-\frac{m}{2}}  \right)G(g_{1},u,v)\int_{M}|u|^{\alpha}|v|^{\beta}uvd\mu_{g_{1}}.
\end{eqnarray*}
Since the eigenfunctions corresponding to $\lambda(t)$ are normalized, thus we get
 \begin{equation*}
\lambda(g_{2})-\lambda(g_{1})\leq\left((1+\epsilon)^{\frac{p+m}{2}} -(1+\epsilon)^{-\frac{m}{2}}  \right) \lambda(g_{1})
\end{equation*}
this completes the proof of Lemma.
\end{proof}
\begin{proposition}\label{pro1p}
Let $(g(t),\phi(t)),\,\,t\in [0,T)$, be a solution of the $(RH)_{\alpha}$ flow  on a closed manifold $M^{m}$ and let $\lambda(t)$  be the first eigenvalue of the $(p,q)$-Laplacian along this flow. Then for  any $t_{0},\,t_{1}\in [0,T)$ and $t_{1}>t_{0}$, we have
 \begin{equation}\label{ppq1}
\lambda(t_{1})\geq\lambda(t_{0})+\int_{t_{0}}^{t_{1}}\mathcal{G}(g(\tau),u(\tau),v(\tau))d\tau
\end{equation}
where
\begin{eqnarray}\label{ppq2}\nonumber
\mathcal{G}(g(t),u(t),v(t))&=&(\alpha+1)\int_{M}\left(\mathcal{S}(\nabla u,\nabla u)+<\nabla u',\nabla u> \right)|\nabla
u|^{p-2}d\mu \\
&&+(\beta+1)\int_{M}\left(\mathcal{S}(\nabla v,\nabla v)+<\nabla v',\nabla v>\right) |\nabla
v|^{q-2}d\mu \\\nonumber
&&-\frac{\alpha+1}{p}\Large\int_{M}|\nabla u|^{p}S d\mu-\frac{\beta+1}{q}\Large\int_{M}|\nabla v|^{q}S d\mu.
\end{eqnarray}
\end{proposition}
\begin{proof}
Assume that
\begin{equation*}
G(g(t),u(t),v(t))=\frac{\alpha+1}{p}\int_{M}|\nabla u(t)|_{g(t)}^{p}d\mu_{g(t)}+\frac{\beta+1}{q}\int_{M}|\nabla v(t)|_{g(t)}^{q}d\mu_{g(t)},
\end{equation*}
at time $t_{1}$ we first let $(u_{1},v_{1})=(u(t_{1}),v(t_{1}))$ be the eigenfunctions for the eigenvalue $\lambda(t_{1})$ of $(p,q)$-Laplacian. We consider the following smooth functions
\begin{equation*}
h(t)=u_{1}\left[\frac{\det[g_{ij}(t_{1})]}{\det[g_{ij}(t)]}\right]^{\frac{1}{2(\alpha+\beta+1)}},\qquad
l(t)=v_{1}\left[\frac{\det[g_{ij}(t_{1})]}{\det[g_{ij}(t)]}\right]^{\frac{1}{2(\alpha+\beta+1)}},
\end{equation*}
along the $(RH)_{\alpha}$ flow. We define
\begin{equation*}
u(t)=\frac{h(t)}{\left(\int_{M}|h(t)|^{\alpha}|l(t)|^{\beta}h(t)l(t)d\mu\right)^{\frac{1}{p}}},\qquad
u(t)=\frac{l(t)}{\left(\int_{M}|h(t)|^{\alpha}|l(t)|^{\beta}h(t)l(t)d\mu\right)^{\frac{1}{q}}}
\end{equation*}
which $u(t),\,v(t)$ are smooth functions under the $(RH)_{\alpha}$  flow, satisfy
\begin{equation*}
\int_{M}|u|^{\alpha}|v|^{\beta}uvd\mu=1,
\end{equation*}
and at time  $t_{1}$, $(u(t_{1}),v(t_{1}))$ are   the eigenfunctions for $\lambda(t_{1})$ of $(p,q)$-Laplacian at time $t_{1}$ i.e.
$\lambda(t_{1})=G(g(t_{1}),u(t_{1}),v(t_{1}))$. If $f$ is a smooth function respect to time  $t$ then along  the $(RH)_{\alpha}$ flow we have
\begin{equation*}
\frac{d}{dt}\left(|\nabla f|^{p}\right)= \frac{p}{2}\left[ \partial_{t}g^{ij}\nabla_{i} f\nabla_{j} f+2g^{ij}\nabla_{i} f'\nabla_{j} f\right]|\nabla f|^{p-2}
\end{equation*}
by (\ref{f1}) we have $ \partial_{t}g^{ij}=2g^{ik}g^{jl}\mathcal{S}_{kl}$, therefore
\begin{equation}\label{ppr1}
\frac{d}{dt}\left(|\nabla f|^{p}\right)=p|\nabla f|^{p-2}\bigg(\mathcal{S}(\nabla f,\nabla f)+<\nabla f',\nabla f>\bigg),
\end{equation}
and
\begin{equation*}
\partial_{t}d\mu=\frac{1}{2}tr_{g}(\partial_{t}g)d\mu=-S d\mu.
\end{equation*}
Since $u(t)$ and $v(t)$ are smooth functions, hence $G(g(t),u(t),v(t))$ is a smooth function with respect to $t$. If we set
\begin{equation}\label{ppp}
\mathcal{G}(g(t),u(t),v(t)):=\frac{d}{dt}G(g(t),u(t),v(t)),
\end{equation}
 then
\begin{eqnarray}\label{ppp1}\nonumber
\mathcal{G}(g(t),u(t),v(t))&=&(\alpha+1)\int_{M}\left(\mathcal{S}(\nabla u,\nabla u)+<\nabla u',\nabla u>\right)|\nabla
u|^{p-2}d\mu\\
&&+(\beta+1)\int_{M}\left(\mathcal{S}(\nabla v,\nabla v)+<\nabla v',\nabla v>\right)|\nabla
v|^{q-2}d\mu\\\nonumber
&&-\frac{\alpha+1}{p}\Large\int_{M}|\nabla u|^{p}S d\mu-\frac{\beta+1}{q}\Large\int_{M}|\nabla v|^{q}S d\mu.
\end{eqnarray}
Taking integration on the both sides of (\ref{ppp}) between $t_{0}$ and $t_{1}$, we conclude that
\begin{equation}\label{ppp2}
G(g(t_{1}),u(t_{1}),v(t_{1}))-G(g(t_{0}),u(t_{0}),v(t_{0}))=\int_{t_{0}}^{t_{1}}\mathcal{G}(g(\tau),u(\tau),v(\tau))d\tau
\end{equation}
where  $t_{0}\in [0,T)$ and $t_{1}> t_{0}$. Noticing $G(g(t_{0}),u(t_{0}),v(t_{0}))\geq \lambda(t_{0})$ and plugin $\lambda(t_{1})=G(g(t_{1}),u(t_{1}),v(t_{1}))$ in (\ref{ppp2}), yields (\ref{ppq1}) and $\mathcal{G}(g(t),u(t),v(t))$ satisfies in (\ref{ppq2}).
\end{proof}
\begin{theorem}\label{tt1}
Let  $(M^{m}, g(t),\phi(t))$ be  a solution of  the  $(RH)_{\alpha}$ flow on the smooth  closed manifold $(M^{m}, g_{0},\phi_{0})$ and  $\lambda(t)$ denotes the
evolution of  the first  eigenvalue under the $(RH)_{\alpha}$ flow.  Suppose that  $k=\min\{p,q\}$ and
\begin{equation}\label{t1}
\mathcal{S}- \frac{1}{k}S g\geq 0 \,\,\text{in} \,\, M^{m}\times [0,T).
\end{equation}
If $S_{\min}(0)\geq 0$, then
$\lambda(t)$ is nondecreasing and differentiable almost everywhere along the   $(RH)_{\alpha}$ flow on $[0,T)$.
\end{theorem}
\begin{proof} Let for any $t_{1}\in[0,T)$,
 $u(t_{1}),\,\,\,v(t_{1})$ be    the eigenfunctions for $\lambda(t_{1})$ of $(p,q)$-Laplacian.  Then $\int_{M}|u(t_{1})|^{\alpha}|v(t_{1})|^{\beta}u(t_{1})v(t_{1})d\mu_{g(t_{1})}=1$ and
\begin{eqnarray}\label{ppq21}\nonumber
\mathcal{G}(g(t_{1}),u(t_{1}),v(t_{1}))&=&(\alpha+1)\int_{M}\left(\mathcal{S}(\nabla u,\nabla u)+<\nabla u',\nabla u> \right)|\nabla
u|^{p-2}d\mu \\
&&+(\beta+1)\int_{M}\left(\mathcal{S}(\nabla v,\nabla v)+<\nabla v',\nabla v>\right) |\nabla
v|^{q-2}d\mu \\\nonumber
&&-\frac{\alpha+1}{p}\Large\int_{M}|\nabla u|^{p}S d\mu-\frac{\beta+1}{q}\Large\int_{M}|\nabla v|^{q}S d\mu.
\end{eqnarray}
Now,  by the time derivative of  the condition
\begin{equation*}
\int_{M}|u|^{\alpha}|v|^{\beta}uvd\mu=1
\end{equation*}
 we  can get
\begin{equation}\label{R13}
(\alpha+1)\int_{M} |u|^{\alpha}|v|^{\beta}u'vd\mu+(\beta+1)\int_{M} |u|^{\alpha}|v|^{\beta}uv'd\mu=\int_{M}S |u|^{\alpha}|v|^{\beta}uvd\mu.
\end{equation}
On the other hand, (\ref{pg1}) and (\ref{pg2}) imply that
\begin{eqnarray}\label{pq4}
\int_{M}<\nabla u',\nabla u>|\nabla u|^{p-2}d\mu=\lambda(t_{1})\int_{M}|u|^{\alpha}|v|^{\beta}u'vd\mu,\\ \label{pq5}
\int_{M}<\nabla v',\nabla v>|\nabla v|^{q-2}d\mu=\lambda(t_{1})\int_{M}|u|^{\alpha}|v|^{\beta}uv'd\mu.
\end{eqnarray}
Therefore from (\ref{R13}), (\ref{pq4}) and (\ref{pq5}) we have
\begin{eqnarray}\nonumber
(\alpha+1)\int_{M}<\nabla u',\nabla u>|\nabla u|^{p-2}d\mu+(\beta+1)\int_{M}<\nabla v',\nabla v>|\nabla v|^{q-2}d\mu\\\label{R14}
=\lambda(t_{1})\int_{M}S |u|^{\alpha}|v|^{\beta}uvd\mu,
\end{eqnarray}
and the replacing  (\ref{R14}) in (\ref{ppq21}), results that
\begin{eqnarray}\nonumber
\mathcal{G}(g(t_{1}),u(t_{1}),v(t_{1}))
&=&\lambda(t_{1})\int_{M}S |u|^{\alpha}|v|^{\beta}uvd\mu+(\alpha+1)\int_{M}\mathcal{S}(\nabla u,\nabla u)|\nabla
u|^{p-2}d\mu\\\label{ppr}
&&-\frac{\alpha+1}{p}\Large\int_{M}|\nabla u|^{p}S d\mu+(\beta+1)\int_{M}\mathcal{S}(\nabla v,\nabla v)|\nabla
v|^{q-2}d\mu\\\nonumber
&&-\frac{\beta+1}{q}\Large\int_{M}|\nabla v|^{q}S d\mu.
\end{eqnarray}
From (\ref{ppr}) and (\ref{t1}) we have
\begin{eqnarray}\label{tpp3}\nonumber
\mathcal{G}(g(t_{1}),u(t_{1}),v(t_{1}))
&\geq&\lambda(t_{1})\int_{M}S |u|^{\alpha}|v|^{\beta}uvd\mu+(\alpha+1)(\frac{1}{k}-\frac{1}{p})\int_{M}|\nabla
u|^{p}Sd\mu\\&&
+(\beta+1)(\frac{1}{k}-\frac{1}{q})\int_{M}|\nabla
v|^{q}Sd\mu.
\end{eqnarray}
Since
\begin{equation*}
\frac{\partial}{\partial t}S=\Delta S+2|\mathcal{S}_{ij}|^{2}+2\alpha|\tau_{g}\phi|^{2}
\end{equation*}
and $|\mathcal{S}_{ij}|^{2}\geq \frac{1}{m}S^{2}$, it follows that
\begin{equation}\label{ss1}
\frac{\partial}{\partial t}S\geq\Delta S+\frac{2}{m}S^{2}.
\end{equation}
The solution to
\begin{equation*}
\frac{d}{dt}y(t)=\frac{2}{m}y^{2}(t),\,\,\,\,\,\,\,\,y(t)=S_{\min}(0),
\end{equation*}
is
\begin{equation}\label{mp}
y(t)=\frac{S_{\min}(0)}{1-\frac{2}{m}S_{\min}(0)t},\,\,\,\,\,\,t\in[0,T'),
\end{equation}
where $T'=\min\{T, \frac{m}{2 S_{\min}(0)}\}$. Using maximum principle to (\ref{ss1}), we get $S\geq y(t)$ along the $(RH)_{\alpha}$ flow. If $S_{\min}(0)\geq 0$ then the nonnegativity of $S$ preserved  along the $(RH)_{\alpha}$ flow. Therefore (\ref{tpp3}) becomes $\mathcal{G}(g(t_{1}),u(t_{1}),v(t_{1}))\geq 0$.
Thus  we get $\mathcal{G}(g(t),u(t),v(t))>0$  in any small enough neighborhood of $t_{1}$. Hence $\int_{t_{0}}^{t_{1}}\mathcal{G}(g(\tau),u(\tau),v(\tau))d\tau>0$ for any $t_{0}<t_{1}$ sufficiently close to $t_{1}$. Since $t_{1}\in [0,T)$ is arbitrary  the  Proposition \ref{pro1p} completes the proof of the first part of theorem. For the differentiability for $\lambda(t)$, since $\lambda(t)$ is increasing  and continues on the interval $[0,T)$, the classical Lebesgue's theorem (see \cite{AMK}),  $\lambda(t)$ is differentiable almost everywhere on $[0,T)$.
\end{proof}
Motivated by the works of X.-D. Cao \cite{XDC1, XDC2} and J. Y. Wu \cite{JY},  similar to proof of Proposition \ref{pro1p} we first introduce a new smooth eigenvalue function along the $(RH)_{\alpha}$ flow  and then we give evolution formula for it. Let $M$ be an $m$-dimensional closed Riemannian manifold and $g(t)$ be a smooth solution  of  the$(RH)_{\alpha}$ flow. Suppose that
\begin{equation*}
\lambda(u,v,t):=\frac{\alpha+1}{p}\int_{M}|\nabla u|^{p}d\mu+\frac{\beta+1}{q}\int_{M}|\nabla v|^{q}d\mu
\end{equation*}
where $u,v$ are smooth functions and satisfy
\begin{equation*}
\int_{M}|u|^{\alpha}|v|^{\beta}uvd\mu=1,\,\,\,\,
\int_{M}|u|^{\alpha}|v|^{\beta}v d\mu=0,\,\,\,\,
\int_{M} |u|^{\alpha}|v|^{\beta}ud\mu=0.
\end{equation*}
The function $\lambda(u,v,t)$ is a smooth eigenvalue function respect to $t$-variable.
If $(u,v)$ are the corresponding eigenfunctions of the first eigenvalue  $\lambda(t_{1})$ then  $\lambda(u,v, t_{1}) =\lambda(t_{1})$. As proof of  Proposition \ref{pro1p}  and Theorem \ref{tt1} we have the following propositions.
\begin{proposition}\label{pro1}
Let  $(M^{m}, g(t),\phi(t))$ be  a solution of  the $(RH)_{\alpha}$ flow
on the smooth  closed manifold $(M^{m}, g_{0},\phi_{0})$. If $\lambda(t)$ denotes the
evolution of the first eigenvalue under the $(RH)_{\alpha}$ flow, then
\begin{eqnarray}\nonumber
\frac{d\lambda}{dt}(u,v,t)|_{t=t_{1}}
&=&\lambda(t_{1})\int_{M}S |u|^{\alpha}|v|^{\beta}uvd\mu+(\alpha+1)\int_{M}\mathcal{S}(\nabla u,\nabla u)|\nabla
u|^{p-2}d\mu\\\label{R7}
&&-\frac{\alpha+1}{p}\Large\int_{M}|\nabla u|^{p}S d\mu+(\beta+1)\int_{M}\mathcal{S}(\nabla v,\nabla v)|\nabla
v|^{q-2}d\mu\\\nonumber
&&-\frac{\beta+1}{q}\Large\int_{M}|\nabla v|^{q}S d\mu.
\end{eqnarray}
where $(u,v)$ is the associated normalized evolving eigenfunctions.
\end{proposition}
Now, we give a variation of $\lambda(t)$ under the normalized $(RH)_{\alpha}$  flow   which is similar to the previous Proposition.
\begin{proposition}\label{pro2}
Let  $(M^{m}, g(t),\phi(t))$ be  a solution of  the  normalized  $(RH)_{\alpha}$ flow
on the smooth  closed manifold $(M^{m}, g_{0},\phi_{0})$. If $\lambda(t)$ denotes the
evolution of the first eigenvalue under the normalized $(RH)_{\alpha}$ flow, then
\begin{eqnarray}\nonumber
\qquad\,\,\frac{d\lambda}{dt}(u,v,t)|_{t=t_{1}}
&=&\lambda(t_{1})\int_{M}S |u|^{\alpha}|v|^{\beta}uvd\mu+(\alpha+1)\int_{M}\mathcal{S}(\nabla u,\nabla u)|\nabla
u|^{p-2}d\mu\\\label{R16}
&&+(\beta+1)\int_{M}\mathcal{S}(\nabla v,\nabla v)|\nabla
v|^{q-2}d\mu-\frac{\beta+1}{q}\Large\int_{M}|\nabla v|^{q}S d\mu\\\nonumber
&&-\frac{\alpha+1}{p}\Large\int_{M}|\nabla u|^{p}S d\mu-\frac{\alpha+1}{m}r(t_{1})\Large\int_{M}|\nabla u|^{p} d\mu\\\nonumber
&&-\frac{\beta+1}{m}r(t_{1})\Large\int_{M}|\nabla v|^{q} d\mu.
\end{eqnarray}
where $(u,v)$ is the associated normalized evolving eigenfunctions.
\end{proposition}
\begin{proof}
In the normalized case, derivative of the integrability condition $\int_{M}|u|^{\alpha}|v|^{\beta}uvd\mu=1$ respect to $t$, results  that
\begin{equation}\label{R17}
(\alpha+1)\int_{M} |u|^{\alpha}|v|^{\beta}u'vd\mu+(\beta+1)\int_{M} |u|^{\alpha}|v|^{\beta}uv'd\mu=- r(t_{1})+\int_{M}S |u|^{\alpha}|v|^{\beta}uvd\mu.
\end{equation}
On the other hand
\begin{equation}\label{R18}
\frac{d}{dt}(d\mu_{t})=\frac{1}{2}tr_{g}(\frac{\partial g}{\partial
t})d\mu=\frac{1}{2}tr_{g}(\frac{2}{m}rg-2\mathcal{S})d\mu=(r-S)d\mu
\end{equation}
hence we can then write
\begin{eqnarray}\label{R19}\nonumber
\frac{d\lambda}{dt}(u,v,t)|_{t=t_{1}}
&=&\frac{\alpha+1}{p}\left(  \frac{p}{2}\int_{M}\left\{-\frac{2}{m}r
|\nabla u|^{2}+2\mathcal{S}(\nabla u ,\nabla u)+2<\nabla u',\nabla
u>\right\}|\nabla u|^{p-2}d\mu\right)
 \\\nonumber
&&+\frac{\beta+1}{q}\left(     \frac{q}{2}\int_{M}\left\{-\frac{2}{m}r
|\nabla v|^{2}+2\mathcal{S}(\nabla v ,\nabla v)+2<\nabla v',\nabla
v>\right\}|\nabla v|^{q-2}d\mu   \right)\\\nonumber
&&+\frac{\alpha+1}{p} \int_{M}|\nabla u|^{p}(r-S)d\mu +\frac{\beta+1}{q}\int_{M}|\nabla v|^{q}(r-S)d\mu,
\end{eqnarray}
but
\begin{eqnarray}\nonumber
\qquad\,\,(\alpha+1)\int_{M}<\nabla u',\nabla u>|\nabla u|^{p-2}d\mu&+&(\beta+1)\int_{M}<\nabla v',\nabla v>|\nabla v|^{q-2}d\mu
\\\label{R20}&=&-\lambda(t_{1}) r(t_{1})+\lambda(t_{1})\int_{M}S |u|^{\alpha}|v|^{\beta}uvd\mu.
\end{eqnarray}
Therefore the proposition is obtained by replacing (\ref{R20}) in previous relation.
\end{proof}

\begin{theorem}\label{tttt}
Let  $(M^{m}, g(t), \phi(t))$ be  a solution of  the $(RH)_{\alpha}$ flow on the smooth  closed manifold $(M^{m}, g_{0}, \phi_{0})$ and  $\lambda(t)$ denotes the
evolution of the first  eigenvalue under the $(RH)_{\alpha}$ flow.  If $k=\min\{p,q\}$,
\begin{equation}\label{t5}
\mathcal{S}- \frac{S}{k}g> 0 \,\,\text{in} \,\, M^{m}\times [0,T)
\end{equation}
and $S_{\min}(0)>0$,
 then the quantity
$\lambda(t)(1-\frac{2}{m}S_{\min}(0)t)^{\frac{m}{2}}$ is nondecreasing along   the  $(RH)_{\alpha}$ flow on $[0,T')$, where $T':=\min\{\frac{m}{2S_{\min}(0)},T\}$.
\end{theorem}
\begin{proof}
According to (\ref{R7}) and (\ref{t5}) we have
\begin{eqnarray}\nonumber
\frac{d\lambda}{dt}(u,v,t)|_{t=t_{1}}
&>&\lambda(t_{1})\int_{M}S |u|^{\alpha}|v|^{\beta}uvd\mu+(\alpha+1)(\frac{1}{k}-\frac{1}{p})\int_{M}|\nabla
u|^{p}Sd\mu\\\label{t3}
&&+(\beta+1)(\frac{1}{k}-\frac{1}{q})\int_{M}|\nabla
v|^{q}Sd\mu.
\end{eqnarray}
If $S_{\min}(0)>0$, then (\ref{mp}) results that the positive of $S$ remains under the $(RH)_{\alpha}$ flow, therefore
\begin{equation}
\frac{d\lambda}{dt}(u,v,t)|_{t=t_{1}}\geq \lambda(t_{1})\frac{S_{\min}(0)}{1-\frac{2}{m}S_{\min}(0)t_{1}}.
\end{equation}
Then in any small enough neighborhood of $t_{1}$ as $I$, we get
\begin{equation}
\frac{d\lambda}{dt}(u,v,t)\geq \lambda(u,v,t)\frac{S_{\min}(0)}{1-\frac{2}{m}S_{\min}(0)t}.
\end{equation}
Integrating the last inequality with respect to $t$ on $[t_{0},t_{1}]\subset I$, we have
\begin{equation}
\ln \frac{\lambda(u(t_{1}),v(t_{1}),t_{1})}{\lambda(u(t_{0}),v(t_{0}),t_{0})}\geq \ln \left(\frac{1-\frac{2}{m}S_{\min}(0)t_{1}}{1-\frac{2}{m}S_{\min}(0)t_{0}} \right)^{-\frac{m}{2}}.
\end{equation}
Since $\lambda(u(t_{1}),v(t_{1}),t_{1})=\lambda(t_{1})$ and $\lambda(u(t_{0}),v(t_{0}),t_{0})\geq \lambda(t_{0})$ we conclude that
\begin{equation}
\ln \frac{\lambda(t_{1})}{\lambda(t_{0})}\geq \ln \left(\frac{1-\frac{2}{m}S_{\min}(0)t_{1}}{1-\frac{2}{m}S_{\min}(0)t_{0}} \right)^{-\frac{m}{2}},
\end{equation}
that is the quantity $\lambda(t)(1-\frac{2}{m}S_{\min}(0)t)^{\frac{m}{2}}$ is nondecreasing in any sufficiently small neighborhood of $t_{1}$. Since  $t_{1}$ is arbitrary, hence  $\lambda(t)(1-\frac{2}{m}S_{\min}(0)t)^{\frac{m}{2}}$  is nondecreasing along the $(RH)_{\alpha}$ flow on $[0,T')$.
\end{proof}
Now, if in  the $(RH)_{\alpha}$ flow, we suppose that $\alpha=0$, then  the $(RH)_{\alpha}$ flow reduce to the Ricci flow and we have the following corollary
\begin{corollary}
Let $g(t),\,\,\,\,\,t\in [0,T)$ be a solution of the Ricci flow on a closed Riemannain manifold $M$ and $\lambda(t)$ denotes the first eigenvalue of the $(p,q)$-Laplacian  (\ref{pq1}). Suppose that $k=\min\{p,q\}$ and $Ric-\frac{R}{k}g\geq 0$ along the Ricci flow.\\
$(1)$ If $R_{\min}(0)\geq0$, then $\lambda(t)$ is nondecreasing along the Ricci flow for any $t\in[0,T)$.\\
$(2)$ If $R_{\min}(0>0$, then the quantity  $(1-R_{\min}(0)t)\lambda(t)$ is nondecreasing along the Ricci flow for any $t\in[0,T')$ where $T'=\min\{T, \frac{1}{R_{\min}(0)}\}$.
\end{corollary}
In dimension two we have
\begin{proposition}
Let $(g(t), \phi(t)),\,\,\,\,t\in [0,T)$ be a solution of the  $(RH)_{\alpha}$ flow on a closed Riemannian surface $M$ and  $\lambda(t)$ denotes the first eigenvalue of the $(p,q)$-Laplacian (\ref{pq1}). \\
$(1)$ Suppose that $Ric\geq \epsilon \nabla\phi\otimes  \nabla\phi$ where $\epsilon\geq 2\alpha\frac{k-1}{k-2}$ and $2\leq k=\min\{p,q\}$.\\
$(1-1)$   If $S_{\min}(0)\geq0$, then $\lambda(t)$ is nondecreasing along the $(RH)_{\alpha}$  for any $t\in[0,T)$.\\
$(1-2)$ If $S_{\min}(0>0$, then the quantity  $(1-S_{\min}(0)t)\lambda(t)$ is nondecreasing along the$(RH)_{\alpha}$  flow on $[0,T')$ where $T'=\min\{T, \frac{1}{S_{\min}(0)}\}$.\\
$(2)$ Suppose that $k=\min\{p,q\}$ and $|\nabla \phi|^{2}\geq k\nabla\phi\otimes\nabla\phi$. \\
$(2-1)$   If $S_{\min}(0)\geq0$, then $\lambda(t)$ is nondecreasing along the $(RH)_{\alpha}$  for any $t\in[0,T)$.\\
$(2-2)$ If $S_{\min}(0>0$, then the quantity  $(1-S_{\min}(0)t)\lambda(t)$ is nondecreasing along the$(RH)_{\alpha}$  flow on $[0,T')$ where $T'=\min\{T, \frac{1}{S_{\min}(0)}\}$.
\end{proposition}
\begin{proof}
In the case of surface, we have $R_{ij}=\frac{R}{2}$. Then
\begin{eqnarray*}
T_{ij}:=\mathcal{S}_{ij}-\frac{S}{k}g_{ij}&=&\frac{R}{2}g_{ij}-\alpha\nabla_{i}\phi\nabla_{j}\phi-\frac{1}{k}(R-\alpha|\nabla\phi|^{2})g_{ij}\\
&=&(\frac{1}{2}-\frac{1}{k})Rg_{ij}-\alpha\nabla_{i}\phi\nabla_{j}\phi+\frac{\alpha}{k}|\nabla\phi|^{2}g_{ij}.
\end{eqnarray*}
For any vector $V=(V^{i})$ we get
\begin{eqnarray*}
T_{ij}V^{i}V^{j}&=&(\frac{1}{2}-\frac{1}{k})R|V|^{2}-\alpha(\nabla_{i}\phi V^{i})^{2}+\frac{\alpha}{k}|\nabla\phi|^{2}|V|^{2}\\
&\geq&(\frac{1}{2}-\frac{1}{k})R|V|^{2}+\alpha(\frac{1}{k}-1)|\nabla\phi|^{2}|V|^{2}.
\end{eqnarray*}
If $Ric\geq \epsilon \nabla\phi\otimes  \nabla\phi$ where $\epsilon\geq 2\alpha\frac{k-1}{k-2}$  then $R\geq \epsilon |\nabla \phi|^{2}$ and
\begin{equation*}
T_{ij}V^{i}V^{j}\geq \left[(\frac{1}{2}-\frac{1}{k})\epsilon+\alpha(\frac{1}{k}-1)\right]|\nabla\phi|^{2}|V|^{2}\geq0.
\end{equation*}
For second case, we have
\begin{eqnarray*}
T_{ij}V^{i}V^{j}&=&R_{ij}V_{i}V^{j}-\alpha\nabla_{i}V^{i}\nabla_{j}V^{j}-\frac{R}{k}|V|^{2}+\frac{\alpha}{k}|\nabla \phi |^{2}|V|^{2}\\
&\geq& R_{ij}V^{i}V^{j}-\frac{\alpha}{k}|\nabla \phi |^{2}|V|^{2}-\frac{R}{k}|V|^{2}+\frac{\alpha}{k}|\nabla \phi |^{2}|V|^{2}=0.
\end{eqnarray*}
Hence the corresponding results follows by Theorems \ref{tt1} and \ref{tttt}.
\end{proof}
When we restrict  the  $(RH)_{\alpha}$ flow to the Ricci flow, we obtain
\begin{corollary}
Let $g(t),\,\,\,\,\,t\in [0,T)$ be a solution of the Ricci flow on a closed Riemannain surface $M$ and $\lambda(t)$ denotes the first eigenvalue of the $(p,q)$-Laplacian  (\ref{pq1}). \\
$(1)$ If $R_{\min}(0)\geq0$, then $\lambda(t)$ is nondecreasing along the Ricci flow for any $t\in[0,T)$.\\
$(2)$ If $R_{\min}(0>0$, then the quantity  $(1-R_{\min}(0)t)\lambda(t)$ is nondecreasing along the Ricci flow for any $t\in[0,T')$ where $T'=\min\{T, \frac{1}{R_{\min}(0)}\}$.
\end{corollary}
\begin{example}
Let $(M^{m}, g_{0})$ be an Einstein manifold i.e. there exists a constant a such that $Ric(g_{0})=ag_{0}$. Assume that $(N,\gamma)=(M,g_{0})$, then $\phi_{0}$ is the identity map. With the  assumption   $g(t)=c(t)g_{0}$,\,\,\,$c(0)=1$ and the fact that $\phi(t)=\phi(0)$  is harmonic map for all $g(t)$, the $(RH)_{\alpha}$ flow reduces to
\begin{equation*}
\frac{\partial c(t)}{\partial t}=-2a+2\alpha,\,\,\,\,\,\,c(0)=1,
\end{equation*}
then the solution of the initial value problem  is  given by
\begin{equation*}
c(t)=(-2a+2\alpha)t+1.
\end{equation*}
Therefore the solution of the $(RH)_{\alpha}$ flow remains Einstein and we have
\begin{eqnarray*}
\mathcal{S}&=&Ric_{g(t)}-\alpha\nabla\phi\otimes\nabla\phi=(a-\alpha)g_{0}=\frac{a-\alpha}{-2(a-\alpha)t+1}g(t),\\
S&=&R-\alpha|\nabla\phi|^{2}=\frac{am}{-2(a-\alpha)t+1}-\alpha\frac{m}{-2(a-\alpha)t+1}=\frac{(a-\alpha)m}{-2(a-\alpha)t+1}.
\end{eqnarray*}
Using  equation (\ref{R7}), we have
\begin{equation*}
\frac{d\lambda}{d t}(u,v,t)|_{t=t_{1}}=\frac{a-\alpha}{-2(a-\alpha)t+1}\left((\alpha+1)\int_{M}|\nabla u|^{p}d\mu+(\beta+1) \int_{M}|\nabla v|^{q}d\mu\right).
\end{equation*}
Now if assume that $p\leq q$ then for $\alpha<a$ and $t_{1}\in[0,T'')$ where $T''=\min\{\frac{1}{2(a-\alpha)}, T\}$, we have
\begin{equation*}
\frac{d\lambda}{d t}(u,v,t)|_{t=t_{1}}\geq \frac{a-\alpha}{-2(a-\alpha)t_{1}+1}\lambda(t_{1}).
\end{equation*}
This results that in any sufficiently small neighborhood of $t_{1}$  as $I_{1}$, we get
\begin{equation*}
\frac{d\lambda}{d t}(u,v,t)\geq \frac{a-\alpha}{-2(a-\alpha)t+1}\lambda(u,v,t).
\end{equation*}
Integrating  the last inequality with respect to $t$ on $[t_{0},t_{1}]\subset I_{1}$ we have
\begin{equation*}
\ln\frac{\lambda(u(t_{1}),v(t_{1}), t_{1})}{\lambda(u(t_{0}),v(t_{0}), t_{0})}\geq \ln\left( \frac{-2(a-\alpha)t_{1}+1}{-2(a-\alpha)t_{0}+1}\right)^{-\frac{p}{2}},
\end{equation*}
but $t_{1}\in[0,T'')$ is arbitrary, $\lambda(u(t_{1}),v(t_{1}),t_{1})=\lambda(t_{1})$ and $\lambda(u(t_{0}),v(t_{0}),t_{0})\geq \lambda(t_{0})$, then $\lambda(t)(-2(a-\alpha)t+1)^{\frac{p}{2}}$ is nondecreasing along the $(RH)_{\alpha}$ flow on $[0,T'')$.
\end{example}

\end{document}